\documentclass[12pt, a4paper]{amsart}
\usepackage{amsfonts}
\usepackage{amssymb}
\usepackage{amsthm}
\usepackage{amsmath}
\usepackage{amscd}
\usepackage[latin2]{inputenc}
\usepackage{t1enc}
\usepackage[mathscr]{eucal}
\usepackage{indentfirst}
\usepackage{graphicx}
\usepackage[all]{xy}
\usepackage{graphics}
\usepackage{pict2e}
\usepackage{epic}
\numberwithin{equation}{section}
\usepackage[margin=2.9cm]{geometry}
\usepackage{epstopdf} 
\usepackage{hyperref}
\usepackage{cleveref}

\hypersetup{
colorlinks = true,
urlcolor   = blue,
citecolor  = blue,
}

\theoremstyle{plain}
\newtheorem{Th}{Theorem}[section]
\newtheorem{Lem}[Th]{Lemma}

\newtheorem{Prop}[Th]{Proposition}

\theoremstyle{definition}

\newtheorem{Rem}[Th]{Remark}
\newtheorem{?}[Th]{Problem}

\newcommand{\diam}{{\rm{diam}}}

\begin{document}

\title{Heegaard distance of the link complements in $S^3$}

\author{Xifeng Jin}

\address{Xifeng Jin\\
Department of Mathematics\\
University at Buffalo--SUNY\\
Buffalo, NY 14260-2900, USA\\
xifengji@buffalo.edu}

\date{\today}
\subjclass[2010]{Primary:  57M25. Secondary:  57M27.}
\keywords{curve complex, disk complex, Heegaard distance}

\begin{abstract} 
We show that, for any integers, $g \geq 3$ and $n \geq 2$, there exists
a link in $S^3$ such that its complement has a  genus $g$ Heegaard splitting with distance $n$. 
\end{abstract}
\maketitle

\section{\textbf{Introduction}}
The Heegaard distance (distance) was introduced by Hempel \cite{Hempel} to measure the complexity of 3-manifolds using Heegaard splittings, which generalizes the notion of Casson and Gordon's strong irreducibility of Heegaard splittings. Using this notion, Hempel showed that for any integers $g \geq 2$ and $n > 0$, there is a 3-manifold with Heegaard splitting of genus $g$ that admits the distance at least $n$, that is, there exist 3-manifolds with high distance. In line with this result, Evans \cite{Evans} and Yoshizawa \cite{Yo} used combinatorial techniques to construct 3-manifolds of high distance. Lustig and Moriah \cite{LM} introduced the fat train tracks to construct  3-manifolds of high distance. 

The topology of the underlying 3-manifold places constraints on the Heegaard distance. If a 3-manifold is Haken, Hartshorn \cite{Hart} showed that its Heegaard distance is bounded above by the double of the genus of an incompressible closed surface. However, Li \cite{Li1} proved that closed non-Haken manifolds admit high Heegaard distance. 

When restricting to the knot complements in $S^3$, Minsky, Moriah and Schleimer \cite{MMS} proved the existence of high distance knot, that is, a knot in $S^3$ whose exterior has a genus $g$ Heegaard splitting of arbitrarily high distance for any $g > 1$. In \cite{CR}, Campisi and Rathbun generalized the result to knots in 3-manifolds. 

On the other hand, the exactness of Heegaard distance of 3-manifolds has been studied in \cite{IJK, John, QZG}. Ido, Jang, Kobayashi \cite{IJK} showed that, for any integers $n \geq 2$ and $g \geq 2$, there exists a genus-$g$ Heegaard splitting of a closed 3-manifold with distance exactly $n$. Johnson \cite{John} proved that, for every pair of positive integers $d \geq 4$, $g \geq 2$ with $d$ even, there is a compact, connected, closed, orientable three-manifold $M$ with both a genus $g$ Heegaard surface $\Sigma$ such that $d(\Sigma)=d$, and a separating, two-sided, closed, embedded, incompressible surface of genus $\frac{1}{2}d$. Qiu, Zou, Guo \cite{QZG} showed that, for any integers $n \geq 1$ and $g \geq 2$, there is a closed 3-manifold $M^n_g$ which admits a distance-$n$ Heegaard splitting of  genus $g$ unless $(g,n) = (2,1)$. 
 
In the setting of bridge splitting of links in 3-manifolds, Ido, Jang, Kobayashi \cite{IJK2} was able to show that, for any integers $n \geq 2$, $g \geq 0$ and $b \geq 1$ except for $(g,b)=(0,1)$ and $(0,2)$, there exists a $(g,b)$-bridge splitting of a link in some 3-manifold with distance exactly $n$. 

Comparable to these results, we show the exactness of Heegaard distance for the link complements in $S^3$. Our result lies in the intersection of high distance knot in \cite{CR, MMS} and the exactness of Heegaard distance of 3-manifolds in \cite{IJK, IJK2, John, QZG}. The ambient 3-manifold is $S^3$ and the link complements in $S^3$ achieve the exact Heegaard distance. 

We start off with two compression bodies of genus $g \geq 3$, each of which is obtained from attaching one 2-handle to the $S \times [0,1]$ along a separating curve. The curve is a meridian of one handlebody in the complement of $S$ embedded in $S^3$ in a standard way. The union of two compression bodies along the common boundary surface $S$ can be embedded in $S^3$. It follows that the disk graph of each compression body contains the unique separating meridian. This gives us a good control of the geodesic realizing the distance between the compression bodies. To achieve the exact distance, we will adopt Ido, Jang and Kobayashi's approach to construct geodesics such that the two ends are meridians. 

\begin{Th}
\label{compression}
 Let $S$ be a closed oriented surface of genus $g \geq 3$ embedded in $S^3$. For any integer $n\geq 0$, there exists a compact oriented 3-manifold $V_0\cup_S W_0$ obtained by the union of two compression bodies $V_0$, $W_0$ with Heegaard distance $n$, and it can be embedded in $S^3$.   
\end{Th}

The meridian realizing the distance divides the genus $g$ surface $S$ into a one-holed torus and a genus $g-1$ surface with one boundary component. Inspired by Minsky, Moriah and Schleimer's work \cite{MMS}, we can attach a genus $g-1$ handlebody that minuses a knot to push the disk graph far away except for the meridian that realizes the distance. 

\begin{Th}
\label{link}
For any integers $g\geq 3$ and $n \geq 2$, there exists a link in $S^3$ such that its complement has a genus $g$ Heegaard splitting with distance $n$. 
\end{Th}

The paper is organized as follows. In Section \ref{section: Preliminaries}, some relevant definitions and results about curve complex, Heegaard splitting and Heegaard distance are given. 
In Section \ref{section: Construction of geodesics}, we construct geodesics in the curve complex using a method from \cite{IJK}, subject to some constraints. Then, we prove Theorem \ref{compression} in Section \ref{section: compression}. In Section \ref{section: link complement}, we prove the main result, Theorem \ref{link} using the result of high distance knot from \cite{MMS} and the argument of exact Heegaard distance from \cite{QZG}.

\section*{Acknowledgements}

The author would like to thank his advisor, Professor William W. Menasco, for many insightful discussions and his continuous support. The author would also like to thank Professor Joan Birman for providing helpful suggestions. Finally, the author thanks the anonymous referee for many helpful comments to improve the readability of the paper.

\section{\textbf{Preliminaries}}
\label{section: Preliminaries}

\subsection{Curve complex}
Let $S=S_{g,b}$ be a compact oriented surface of genus $g$ and $b$ boundary components. The \emph{complexity} of surface $S$ is defined as $\xi(X)=3g+b-3$. A simple closed 
curve in $S$ is \emph{essential} if it does not bound a disk or an annulus in $S$. 

The \emph{curve complex}, denoted by $\mathcal{C}(S)$, is a simplicial complex such that 
each vertex is represented by the isotopy class of an essential simple closed 
curve, and $n+1$ vertices form an $n$-simplex of $\mathcal{C}(S)$ if their 
representatives can be realized disjointly. It was introduced by Harvey \cite{Harvey} to study the mapping 
class group. The information of curve complex is encoded in its 1-skeleton, which is called \emph{curve graph}. Throughout, we only 
consider the curve graph instead of curve complex, and we use the same 
notation $\mathcal{C}(S)$ for the curve graph. 

The \emph{arc and curve complex} $\mathcal{AC}(S)$ of 
a compact oriented surface $S$ with boundary can be defined similarly. Each 
vertex of $\mathcal{AC}(S)$ is the isotopy class of an essential properly 
embedded arc (each endpoint is allowed to move freely in its boundary) or an essential simple closed curve on $S$, and $n+1$ vertices 
form an $n$-simplex of 
$\mathcal{AC}(S)$ if their representatives can be realized disjointly. The 
notation $\mathcal{AC}(S)$ will also be used as the 1-skeleton of the arc and 
curve complex, and it is called the \emph{arc and curve graph}. 

For any two vertices $x$ and $y$ in $\mathcal{C}(S)$, the \emph{distance} $d_{\mathcal{C}(S)}(x,y)$ 
is the minimal number of edges in $\mathcal{C}(S)$ joining $x$ and $y$. For any 
two subsets $A$ and $B$, define 
$d_{\mathcal{C}(S)}(A,B)=\text{min}\{d_{\mathcal{C}(S)}(x,y)|x\in A,y\in B\}$. A 
\emph{geodesic} in the curve graph $\mathcal{C}(S)$ is a sequence of vertices 
$\Gamma=(\gamma_i)$ such that $d_{\mathcal{C}(S)}(\gamma_i,\gamma_j)=|i-j|$ for 
all $i,j$. These notions can be defined on the arc and curve graph $\mathcal{AC}(S)$ similarly. 
A metric space is \emph{$\delta$-hyperbolic}, 
if for each geodesic triangle in the metric space, each side lies in the 
$\delta$-neighborhood of the other two sides. 

\begin{Th} \emph{(Masur-Minsky \cite{MM1} Theorem 1.1)}\label{infinite diameter}
Let $S$ be a compact oriented surface with $\xi(S) \geq 1$, the 
curve graph $\mathcal{C}(S)$ is a $\delta$-hyperbolic metric space with infinite 
diameter for some $\delta$, where $\delta$ depends on the surface.   
\end{Th}

\subsection{Disk graph}
Let $M$ be a compact oriented 3-manifold, and suppose the closed 
oriented surface $S_g$ is a boundary component of $M$. Denote  
$$\mathcal{D}(M,S_g)=\{ [\partial D]: (D,\partial D)\subset (M,S_g)\ \text{is an 
essential disk}\} \subset \mathcal{C}(S_g),$$ where $[*]$ denotes the isotopy class of a simple closed curve $*$ in the surface $S_g$. 
As a full subgraph of $\mathcal{C}(S_g)$, $\mathcal{D}(M,S_g)$ is called the \emph{disk graph}. It is the 1-skeleton of disk complex introduced by McCullough \cite{MC}.

\subsection{Heegaard splitting}
A \emph{handlebody $V$ of genus $g$} is a 3-manifold homeomorphic to a regular neighborhood 
of a connected graph in $S^3$, whose boundary is a closed oriented surface of genus $g$. A \emph{Heegaard splitting} of genus $g$ for a 
closed 3-manifold $M$ is $M = V \cup_{S} W$, where $V$ and $W$ are two handlebodies 
of genus $g$. The common boundary $S = \partial V= \partial W$ is a closed 
oriented surface of genus $g$ in $M$, which is called the \emph{Heegaard 
surface} of the Heegaard splitting. 

A \emph{compression body} $V$ is a compact oriented 3-manifold obtained from 
$S\times [0,1]$ and a 0-handle $B$ by attaching 1-handles to $S\times \{0\}\cup \partial B$, where $S$ is a closed oriented
surface, but it is not necessarily connected or nonempty. The \emph{negative \emph{(}inner\emph{)} boundary}  is $S\times\{1\}\subset V$ and the \emph{positive \emph{(}outer\emph{)} boundary} $\partial _+V=\partial V-\partial_{-}V$. By convention, a handlebody is a compression body with $\partial_{-}V=\varnothing$. 

Dually, a \emph{compression body} $V$ is obtained by attaching some 2-handles to 
$S\times \{1\}$ as the boundary of $S\times [0,1]$, and 3-handles to cap off any 
2-spheres created by the attachment of the 2-handles. The \emph{positive \emph{(}outer\emph{)} boundary} $\partial_+V$ is the boundary component $S\times \{0\}$.  
The \emph{negative \emph{(}inner\emph{)} boundary} $\partial_{-}V=\partial V-\partial_{+}V$. If $\partial_{-}V=\varnothing$, then it is a handlebody. 

For a compact oriented 3-manifold $M$ with boundary, the \emph{Heegaard 
splitting} of genus $g$ is $M=V\cup_S W$, where $V$ and $W$ are compression 
bodies. $S=\partial_+V=\partial_+W$ is a closed oriented surface of 
genus $g$ embedded in $M$.

A Heegaard splitting $M=V\cup_{S}W$ is called \emph{reducible} if there exists a 
pair of essential embedded disks $(D,D')\subset (V,W)$ with $\partial 
D=\partial D'\subset S$. Otherwise, it is $\emph{irreducible}$. A 
Heegaard splitting is called \emph{weakly reducible} if there exists a pair of 
essential embedded disks $(D,D')\subset (V,W)$ with $\partial D\cap 
\partial D'=\varnothing$. Otherwise, it is \emph{strongly irreducible}. 
The \emph{weak reducibility} and \emph{strong reducibility} were introduced 
by Casson and Gordon \cite{CG}.

\subsection{Heegaard distance}
An essential curve on $\partial_+ V$ is a \emph{meridian} of the compression body 
$V$ if it bounds an essential disk in $V$. If the closed oriented surface 
$S=\partial_+V$, then the subcomplex $\mathcal{D}(V,S)$ of $\mathcal{C}(S)$ 
spanned by the meridians is the disk graph of the compression body. In the 
case of compression body and handlebody, the disk graph is denoted by $\mathcal{D}(V)$ instead of 
$\mathcal{D}(V,S)$ for simplicity.

For a Heegaard splitting $M = V \cup_S W$ with the Heegaard surface $S$, Hempel \cite{Hempel} defined the 
\emph{Heegaard distance} (\emph{distance}) of a Heegaard splitting to be
$$d_{\mathcal{C}(S)}(V,W): =d_{\mathcal{C}(S)}(\mathcal{D}(V),\mathcal{D}(W))$$
The Heegaard splitting $M=V\cup_S W$ is \emph{reducible} if and only if 
$d_{\mathcal{C}(S)}(V,W)=0$, and it is \emph{weakly reducible} if and only if 
$d_{\mathcal{C}(S)}(V,W)\leq 1$.

\subsection{Subsurface projection}

A subsurface $X$ is an \emph{essential subsurface} of $S$, 
if $X$ is a compact connected oriented proper subsurface such that each component of $\partial X - \partial S$ is essential in $S$. Suppose $\xi(X) \geq 1$, we define a map $\pi_A:\mathcal{C}_0(S)\rightarrow 
\mathcal{P}(\mathcal{AC}_0(X))$, where $\mathcal{P}(\mathcal{AC}_0(X))$ is the power 
set of the vertices $\mathcal{AC}_0(X)$. Take any $\alpha$ in 
$\mathcal{C}_0(S)$, then consider the representative of $\alpha$ such that it 
intersects $X$ minimally, $\pi_A(\alpha)$ is the set of all isotopy classes of 
$\alpha \cap X$ relative to the boundary of $X$. Note that $\pi_A(\alpha)$ is empty when $\alpha$ can be 
realized disjointly from $X$. We say $\alpha$ \emph{cuts} $X$ if $\alpha \cap 
X \neq \varnothing$, and $\alpha$ \emph{misses} $X$ if $\alpha \cap X = \varnothing$.

There is a natural way to send them back to the curves in $\mathcal{C}_0(X)$. We can define $\pi_0: \mathcal{P}(\mathcal{AC}_0(X))\rightarrow \mathcal{C}_0(X)$ as 
follows. If $\alpha$ is in $X$, then $\pi_A(\alpha) = \alpha$ in $\mathcal{AC}_0(X)$ and $\pi_0( \pi_A(\alpha)) = \alpha$ in $\mathcal{C}_0(X)$. Otherwise, 
$\pi_A(\alpha)=\{\alpha_1,\alpha_2,\cdots,\alpha_n\}$ is a collection of isotopy 
classes of essential properly embedded arcs in $X$. The set 
$\pi_0(\pi_A(\alpha))$ is the isotopy classes of the essential components of 
$\partial N(\alpha_i\cup \partial X)$ in $X$, where
$N(\alpha_i\cup X)$ is a regular neighborhood of $\alpha_i\cup \partial X$ in 
$X$. The composition 
$\pi_0\circ \pi_A = \pi_X: \mathcal{C}_0(S) \rightarrow \mathcal{C}_0(X)$ is called the \emph{subsurface projection}.

\begin{Lem}\emph{(Masur-Minsky \cite{MM2} Lemma 2.2)}
Suppose that the complexity $\xi(X)>1$, and $d_{\mathcal{AC}(X)}(\alpha,\beta)\leq 1$ for $\alpha, \beta \in \mathcal{AC}(X)$, then
$d_{\mathcal{C}(X)}(\pi_0(\alpha),\pi_0(\beta))\leq 2$.  
\end{Lem}

 The notation $d_X(A)$ is the diameter of $\pi_X(A)$ in $\mathcal{C}(X)$ for a subset $A \subset \mathcal{C}_0(S)$ and $d_Y(A,B)$ is the diameter of $\pi_X(A) \cup \pi_X(B)$ in $\mathcal{C}(X)$ for two subsets $A$ and $B$ in $\mathcal{C}_0(S)$. The subsurface projection $\pi_X$ has a strong contraction property. 

\begin{Th}\emph{(Bounded Geodesic Image Theorem, Masur-Minsky \cite{MM2}, Theorem 
3.1)}\label{Bounded}
 Let $X$ be an essential subsurface of $S$ with $\xi(X)\neq 0$ and let 
$\Gamma=(\gamma_i)_{i\in I}$ be a geodesic in $\mathcal{C}(S)$. If each 
$\gamma_i$ cuts $X$, then there is a constant $M$ depending only on the surface 
so that $d_{\mathcal{C}(X)}(\pi_X(\Gamma))\leq M$. 
\end{Th}

The constant $M$ can be taken to be independent of the surface 
\cite{Webb}.

Let $F$ be a boundary component of a compact orientable 3-manifold $M$.
A simple closed curve $\gamma$ on an oriented closed surface $F$ is \emph{disk-busting} if it intersects every simple closed curve in the disk graph $\mathcal{D}(M, F)$, that is, $F - \gamma$ is incompressible in $M$. If $M$ is an $I$-bundle over a compact surface $S$ with boundary $\partial S$, then the boundary  $\partial M$ can be decomposed into two parts, vertical boundary and horizontal boundary. The vertical boundary $\partial_v M$ is the $I$-bundle restricted to the $\partial S$. The horizontal boundary $\partial_h M = \overline{\partial M - \partial_v M}$ is the portion of $\partial M$ that is transverse to the $I$-fibers. 

The following theorem of the subsurface projection of disk complex was proved independently by Li \cite{Li2}, Masur and Schleimer \cite{MS}. 
\begin{Th}\emph{(Li \cite{Li2} Theorem 1, Masur-Schleimer \cite{MS} Theorem 12.1)}\label{Bounded disk}
Let $M$ be a compact orientable and irreducible 3-manifold and $F$  a component of $\partial M$. Suppose $\partial M - F$ is incompressible in $M$. Let $\mathcal{D}$ be the disk complex for $F$. Let $S$ be a compact connected subsurface of $F$ and suppose that every component of $\partial S$ is disk-busting. Then either
\begin{enumerate}
\item $M$ is an $I$-bundle over a compact surface, and $S$ is a component of the horizontal boundary of this $I$-bundle, and the vertical boundary of this $I$-bundle is a single annulus, or
\item the image $\pi_A(\mathcal{D})$ of the disk complex lies in a ball of radius 3 in $\mathcal{AC}(S)$, in particular, $\pi_A(\mathcal{D})$ has diameter at most 6 in $\mathcal{AC}(S)$. Moreover, $\pi_S(\mathcal{D})$ has diameter at most 12 in $\mathcal{C}(S)$.
\end{enumerate}
\end{Th}

In the end, we recall a main property of pseudo-Anosov maps.

\begin{Th}\emph{(Masur-Minsky \cite{MM1} Proposition 4.6)}\label{MM_pS}
For a surface $S_{g,b}$ with the complexity $\xi(S_{g,b})>1$, there exists $c>0$ 
such that, for any pseudo-Anosov $h\in Mod(S_{g,b})$, any $\gamma \in 
\mathcal{C}(S_{g,b})$ and any $n\in \mathbb{Z}$,
$$d_{\mathcal{C}(S_{g,b})}(h^n(\gamma),\gamma)\geq c|n|.$$
\end{Th}

A well known method to construct pseudo-Anosov maps on $S_{g,b}$ is Thurston's construction. A 
\emph{filling pair}  on $S_{g,b}$ is a pair 
of curves $\alpha$ and $\beta$ such that each complement of $S_{g,b}-\alpha\cup \beta$ is either a disk or an annulus. 
\begin{Th}\emph{(Thurston \cite{Thurston})}\label{Thurston}
If $\alpha$ and $\beta$ is a filling pair on $S_{g,b}$, then the composition of Dehn twists $T_{\alpha}\circ T^{-1}_{\beta}$ is a pseudo-Anosov map.
\end{Th}

\section{\textbf{Construction of geodesics}}
\label{section: Construction of geodesics}

In this section, we will construct geodesics of exact distance in the curve graph 
$\mathcal{C}(S_g)$ in certain conditions. First, we state two criterions that have been used to extend the geodesics from \cite{IJK}.

\begin{Prop}\emph{(Ido-Jang-Kobayashi \cite{IJK} Proposition 4.1)}
\label{even}
For an integer $n\geq 4$, suppose that $[\alpha_0, \alpha_1, \cdots, \alpha_n]$ 
is a path in the curve graph $\mathcal{C}(S_g)$ satisfying the following:\\
\emph{(1)} $[\alpha_0,\cdots, \alpha_{n-2}]$ and 
$[\alpha_{n-2},\alpha_{n-1},\alpha_n]$ are geodesics in $\mathcal{C}(S_g)$.\\
\emph{(2)} 
$\diam_{\mathcal{C}(X_{n-2})}(\pi_{X_{n-2}}(\alpha_{n-4}),\pi_{X_{n-2}}(\alpha_n))\geq 
4n$, where $X_{n-2}=\overline{S_g - N(\alpha_{n-2})}$.\\
Then $[\alpha_0,\alpha_1,\cdots,\alpha_n]$ is a geodesic in $\mathcal{C}(S_g)$.
\end{Prop}

\begin{Prop}\emph{(Ido-Jang-Kobayashi \cite{IJK} Proposition 4.4)}
\label{odd}
 For an integer $n\geq 3$, suppose that $[\alpha_0, \alpha_1, \cdots, \alpha_n]$ 
is a path in the curve graph $\mathcal{C}(S_g)$ satisfying the following:\\
 \emph{(1)} $[\alpha_0,\cdots, \alpha_{n-1}]$ and 
$[\alpha_{n-2},\alpha_{n-1},\alpha_n]$ are geodesics in $\mathcal{C}(S_g)$.\\ 
 \emph{(2)} the union $\alpha_{n-2}\cup \alpha_{n-1}$ is nonseparating in $S_g$, 
and 
$\diam_{\mathcal{C}(S'_g)}(\pi_{S'_g}(\alpha_0),\pi_{S'_g}
(\alpha_n))\geq 2n$, where $S'_g=\overline{S_g - N(\alpha_{n-2}\cup 
\alpha_{n-1})}$.\\
Then $[\alpha_0,\alpha_1,\cdots,\alpha_n]$ is a geodesic in $\mathcal{C}(S_g)$.
\end{Prop}

With these two propositions, we will be able to construct the geodesics in the $\mathcal{C}(S_{g \geq 3})$, where $S_{g \geq 3}$ is the boundary of a handlebody. The proof given below is similar to the construction of the geodesics in \cite{IJK}. The difference is that we also need to take account of the meridians when we choose the curves.

\begin{Lem}
\label{lem: geodesics}
Let $V$ be one handlebody in the standard Heegaard splitting $S^3 = V \cup_S W$ of $S^3$. For 
any positive integer $n$, there exists a geodesic $[\alpha_0,\alpha_1,\cdots,\alpha_n]$ 
in the curve graph $\mathcal{C}(S)$ such that  
$|\alpha_{i-2} \cap \alpha_i| = 1$ for any positive even number $i \leq n$. The curve $\alpha_k$ is a meridian of $V$ if $k$ is divisible by 4. Moreover, $\alpha_k$ is a meridian of $V$ if $k$ is odd and $k < n$. If $n$ is odd, then $|\alpha_n \cap \alpha_{n-2}| = 1$.
\end{Lem}

\begin{figure}[ht]
\vspace{-0.2cm}
\scalebox{.20}{\includegraphics[origin=c]{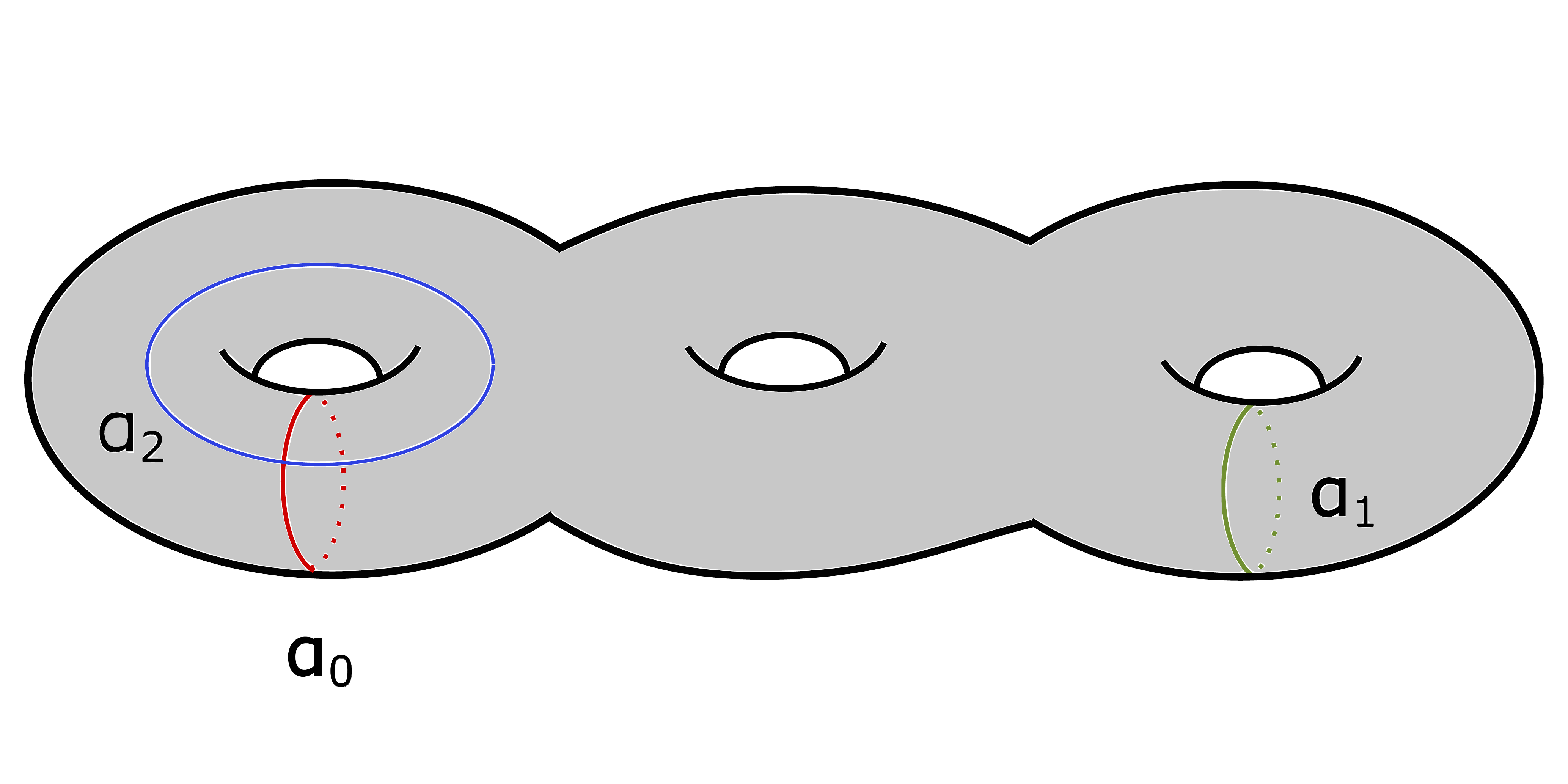}}
\vspace{-0cm}
\caption{A geodesic [$\alpha_0$,$\alpha_1$,$\alpha_2$] on the boundary surface of a genus $g \geq 3$ handlebody $V$ with $|\alpha_0 \cap \alpha_2| = 1$; $\alpha_0$ and $\alpha_1$ are meridians of $V$.}
\label{geodesics}
\end{figure}

\begin{proof}
 First, let us consider the case when $n$ is even with $n \geq 4$. Let $\alpha_0$, $\alpha_1$ 
and $\alpha_2$ be nonseparating simple closed curves on $S$ such that 
$\alpha_1$ is disjoint from $\alpha_0$ and $\alpha_2$ and $|\alpha_0\cap 
\alpha_2| = 1$, see Fig.~\ref{geodesics}. Notice that $\alpha_0$ and $\alpha_1$ 
are meridians of $V$ and $[\alpha_0,\alpha_1,\alpha_2]$ is a 
geodesic of length 2 in $\mathcal{C}(S)$. Let $X_2=\overline{S-N(\alpha_2)}$ be the closure of the complement of regular neighborhood of $\alpha_2$, then one can choose 
a partial pseudo-Anosov map $\varphi_2: S\rightarrow S$ such that $\varphi_2$ 
fixes $\alpha_2$ and 
$\diam_{\mathcal{C}(X_2)}(\pi_{X_2}(\alpha_0), \pi_{X_2}(\varphi_2(\alpha_0)))\geq 4n$. The 
existence of such partial pseudo-Anosov map is justified by Theorem \ref{MM_pS}. Furthermore, one can choose a pseudo-Anosov map that can be extended over the handlebody $V$.

\begin{figure}[ht]
\vspace{-0.2cm}
\scalebox{.20}{\includegraphics[origin=c]{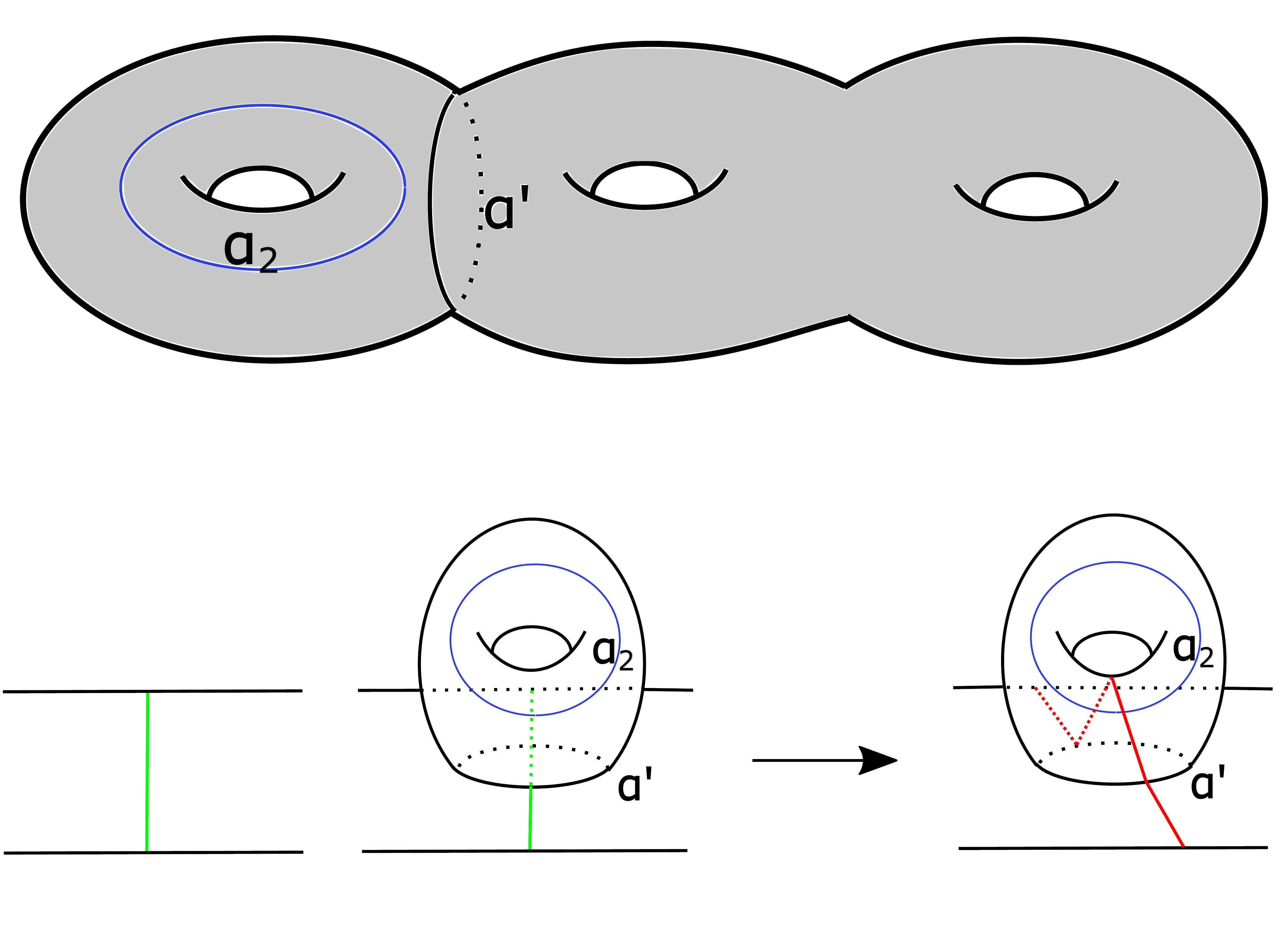}}
\vspace{-0cm}
\caption{Curve $\alpha'$ intersects one subarc of a meridian in a filling pair of meridians in $V_{g-1}$. The subarc of a meridian in green has been replaced by the subarc in red. The resulting curve is a meridian that intersects $\alpha_2$ once. }
\label{nonmeridian}
\end{figure}

Since $|\alpha_0 \cap \alpha_2| = 1$ and $\alpha_0$ is meridian, then the boundary curve $\alpha' = \partial N(\alpha_2 \cup \alpha_0)$ is a separating meridian.
The curve $\alpha'$ separates the handlebody into one solid torus and a handlebody of genus $g-1$. On the handlbody of genus $g-1 \geq 2$, there exists a filling pair, as illustrated in Fig.~\ref{nonmeridian}.  One can choose a filling pair in the positive boundary $S_{g-1}$ such that only one subarc of one meridian intersects $\alpha'$ twice. To make the pair fill the subsurface $X_2$, one can replace the subarc of the meridian with the arc that passes over the $\alpha_2$ once. The resulting new curve is denoted as $\gamma$, then $\partial N(\alpha_2 \cup \gamma)$ is a meridian. Together with the other meridian, they are a filling pair of the subsurface $X_2$. 

Let the filling pair of meridians of $X_2$ be $\alpha$, $\beta$, and define the map 
$\varphi_2 = T_{\alpha}\circ T^{-1}_{\beta}$. By Thurston's construction, $\varphi_2$ is a pseudo-Anosov map, and it can be extended over the handlebody $V$. Iterate it if needed to satisfy $\diam_{\mathcal{C}(X_2)}(\pi_{X_2}(\alpha_0), \pi_{X_2}(\varphi_2(\alpha_0)))\geq 4n$.

Both $\alpha_3=\varphi_2(\alpha_1)$ and $\alpha_4=\varphi_2(\alpha_0)$ are meridians, and 
$[\alpha_2,\alpha_3,\alpha_4]$ is a geodesic of length 2 with $|\alpha_2\cap 
\alpha_4|=1$. By Proposition \ref{even}, 
$[\alpha_0,\alpha_1,\alpha_2,\alpha_3,\alpha_4]$ is a geodesic of length 4.

\begin{figure}[ht]
\vspace{-0.2cm}
\scalebox{.20}{\includegraphics[origin=c]{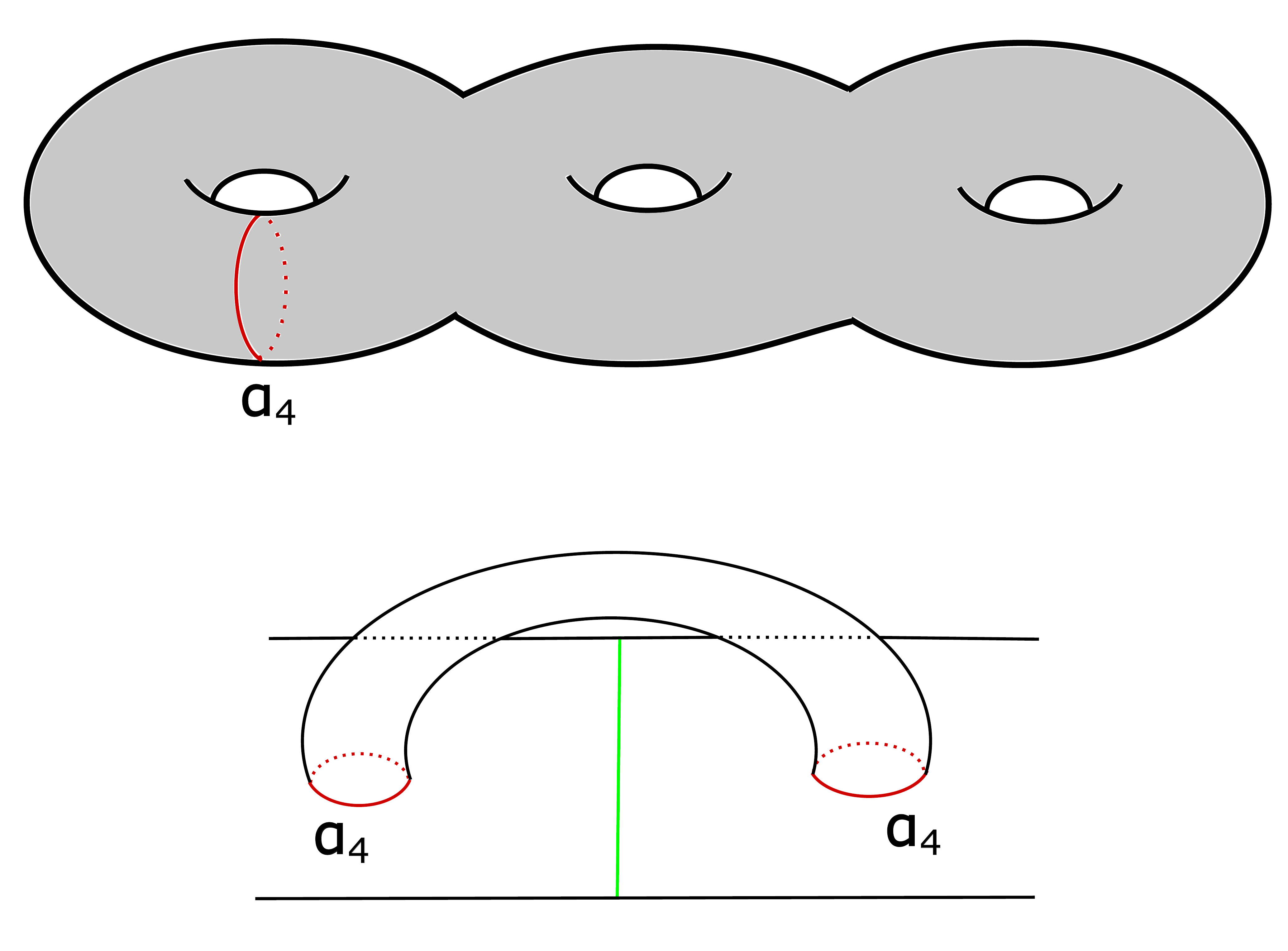}}
\vspace{-0cm}
\caption{ Remove two open disks from two distinct disk complements of a filling pair of meridians on $S_{g-1}$ to obtain the subsurface.}
\label{meridian}
\end{figure}

The curve $\alpha_4$ is a nonseparating meridian of $V$, and the disk bounded by $\alpha_4$ cuts $V$ into a handlebody of genus $g-1$, as illustrated in Fig.~\ref{meridian}. Again, on the genus $g-1$ handlebody, there exists a filling pair of meridians, the subsurface $X_4=\overline{S - N(\alpha_4)}$ can be obtained from removal of two disks in two disk complements of a filling pair of meridians. Similar as before, one can construct a pseudo-Anosov $\varphi_4$ on $X_4$, and it can be extended over $V$. Let $\alpha_5 = \varphi_4(\alpha_3)$ and $\alpha_6 = \varphi_4(\alpha_2)$, then $[\alpha_0,\alpha_1,\cdots,\alpha_5,\alpha_6]$ is a geodesic of length 6.

Continue in this way, 
we can construct a geodesic $[\alpha_0,\alpha_1,\cdots,\alpha_n]$ of 
even length. Assume that $[\alpha_0,\alpha_1,\dots,\alpha_i]$ is a geodesic 
with $|\alpha_{i-2}\cap \alpha_i|=1$ for even $i< n$. Let 
$X_i=\overline{S-N(\alpha_i)}$, then we can take a partial pseuso-Anosov map 
$\varphi_i:S\rightarrow S$ such that $\varphi_i$ fixes $\alpha_i$ and 
$\diam_{\mathcal{C}(X_i)}(\pi_{X_i}(\alpha_{i-2}),\pi_{X_i}(\varphi_i(\alpha_{i-2})))\geq 
4n$. The pseudo-Anosov map $\varphi_i$ can be chosen to be able to be extended 
over the handlbody $V$ as we did before. If $i$ is not divisible by 4, then $\alpha_i$ is not a meridian. The $\alpha_{i-2}$ is a meridian and $|\alpha_i \cap \alpha_{i-2}| = 1$, so the boundary curve $\partial N(\alpha_i \cup \alpha_{i-2})$ is meridian, and it bounds a disk that cuts $V$ into a solid torus and a handlebody of genus $g-1$. Hence, we end up with the case in Fig.~\ref{nonmeridian}. Similarly, the other case is illustrated in Fig.~\ref{meridian}.

Denote 
$\alpha_{i+1}=\varphi_i(\alpha_{i-1})$ and 
$\alpha_{i+2}=\varphi_i(\alpha_{i-2})$ , then 
$[\alpha_{i},\alpha_{i+1},\alpha_{i+2}]$ is a geodesic of length 2. Again, by Proposition \ref{even}, the extended path 
$[\alpha_0,\alpha_1,\cdots,\alpha_{i+1},\alpha_{i+2}]$ is a geodesic with 
$|\alpha_i\cap \alpha_{i+2}|=1$. The 
construction yields a geodesic 
$[\alpha_0,\alpha_1,\cdots,\alpha_{n-1},\alpha_{n}]$ of length $n$ being even. The curves are all meridians except for the ones $\alpha_{4m-2}$, where $m$ is a positive integer. Moreover, 
$|\alpha_0\cap \alpha_2|=1$ and $|\alpha_{i-2}\cap \alpha_i|=1$ for any positive even number $i \leq n$.

Next, we discuss the geodesics of odd length. Suppose that $n-1$ is even with $n \geq 3$, let $[\alpha_0,\alpha_1,\cdots,\alpha_{n-1}]$ be a geodesic 
constructed as the previous case. Since $\alpha_0$, $\alpha_1$ and $\alpha_2$ are nonseparating, then all curves in the geodesic are 
non-separating by construction. By construction, $|\alpha_{n-3}\cap \alpha_{n-1}|=1$, 
and $\alpha_{n-2}$ is disjoint from $\alpha_{n-3}\cup \alpha_{n-1}$, then 
$\alpha_{n-2}\cup \alpha_{n-1}$ is nonseparating. Let $S' = \overline{S-N(\alpha_{n-2} \cup \alpha_{n-1})}$, Theorem \ref{infinite diameter} states that the curve graph $\mathcal{C}(S')$ has infinite diameter. Then there exists $\gamma'$ in $S'$ with $d_{\mathcal{C}(S')}(\gamma', \pi_{S'}(\alpha_0)) > 2n + 2$. Since the genus $g \geq 3$, we can find $\gamma''$ in $S'$ with $d_{\mathcal{C}(S')}(\gamma'', \gamma') \leq 2$ and that $\gamma''$ cuts off a pair of pants $P$ with $\partial N(\alpha_{n-2}) \subset \partial P$.

\begin{figure}[ht]
\vspace{-0.2cm}
\scalebox{.20}{\includegraphics[origin=c]{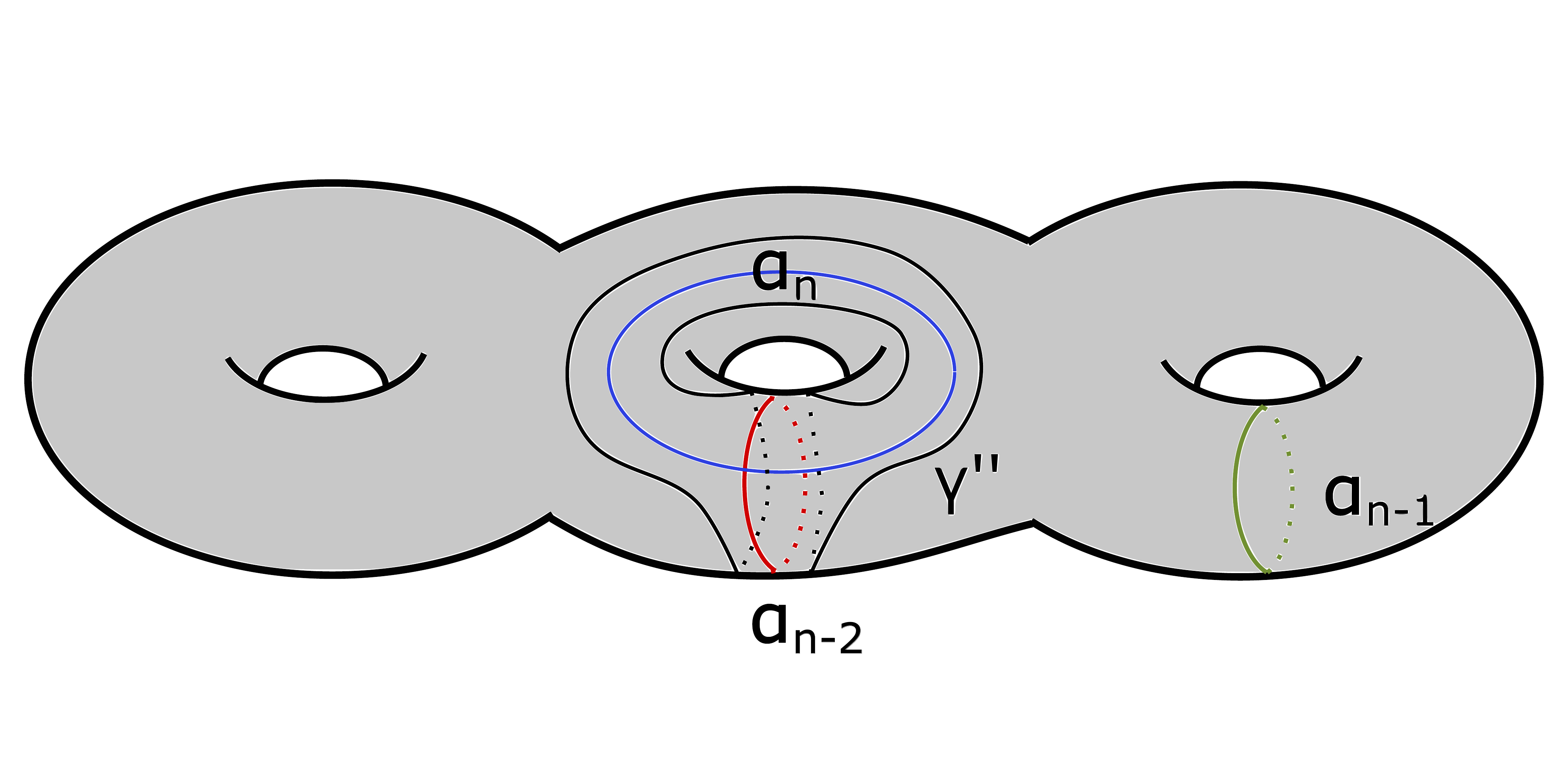}}
\vspace{-0cm}
\caption{Construction of $\gamma$, $\gamma''$ and $\alpha_{n}$.}
\label{construction}
\end{figure}

There exists a curve (not a meridian) $\alpha_n\in \mathcal{C}(S)$ with $|\alpha_{n}\cap \alpha_{n-2}| = 1$, $\alpha_n \cap \alpha_{n-1}=\varnothing$ and 
$\pi_{S'}(\alpha_n)=\gamma''$. By the triangle inequality, we have  
\begin{eqnarray*}
\diam_{\mathcal{C}(S')}(\pi_{S'}(\alpha_0), \pi_{S'}(\alpha_n)) &= &\diam_{\mathcal{C}(S')}(\pi_{S'}(\alpha_0),\gamma'') \\
  & \geq &
\diam_{\mathcal{C}(S')}(\pi_{S'}(\alpha_0),\gamma')-d_{\mathcal
{C}(S')}(\gamma'',\gamma') \\
& > & (2n+2)-2=2n.
\end{eqnarray*}
 Since $\alpha_{n}\cap \alpha_{n-2}=1$ and $\alpha_n\cap \alpha_{n-1}=\varnothing$, 
then $[\alpha_{n-2},\alpha_{n-1},\alpha_{n}]$ is a geodesic in $\mathcal{C}(S)$, see Fig.~\ref{construction}. By Proposition \ref{odd}, 
$[\alpha_0,\alpha_1,\cdots,\alpha_{n-1},\alpha_{n}]$ is a geodesic in 
$\mathcal{C}(S)$.

\end{proof}

\section{\textbf{Proof of Theorem \ref{compression}}}
\label{section: compression}
Let $S$ be a closed oriented surface, and $a$ is a simple closed curve on $S$. The compression body $S[a]$ is obtained from $S \times [0,1]$ by attaching a 2-handle along $a$ onto the boundary $S \times \{1\}$. The disk complex of $S[a]$ is described as follows.

\begin{Prop}\emph{(Biringer-Vlamis \cite{BV} Proposition 2.5)}\label{one meridian}
Suppose that $S$ is a closed, orientable surface and $a$ is a simple closed curve on $S$. If $S$ is a torus or $a$ is separating, 
$$\mathcal{D}(S[a])=\{a\},$$
while if the genus $g(S) \geq 2$ and $a$ is nonseparating, then 
\begin{eqnarray*}
\mathcal{D}(S[a])&=&\{a\} \cup \{B(a,b): b \in \mathcal{C}(S), i(a,b)=1\}\\
                           &=&\{a\} \cup \{\partial T: T \subset S, \text{a punctured torus with}\ a \subset T\}.
\end{eqnarray*}
\end{Prop}
The $B(a,b)$ is the \emph{band sum} of $a$ and $b$, that is, $B(a,b) = \partial N(a \cup b)$. Let $S^3=V\cup_{S} W$ be the standard genus $g \geq 3$ Heegaard splitting. Using Lemma \ref{lem: geodesics},  we can find two meridians in the two handlebodies $V$ and $W$ with exact distance. The proof of Theorem \ref{compression} is followed by a sequence of propositions.

\begin{Prop}
 For any integer $g \geq 3$, there exists a genus $g$ Heegaard splitting $V_0 \cup_S W_0$ with distance $n-2$, where $n$ is any positive integer divisible by 4. The $V_0$ and $W_0$ are compression bodies and $V_0 \cup_S W_0$ can be embedded in $S^3$.
\end{Prop}

\begin{proof}
 Suppose a positive integer $n$ is divisible by 4, let $[\alpha_0,\alpha_1,\cdots,\alpha_{n-1},\alpha_{n}]$ be a geodesic  constructed in Lemma \ref{lem: geodesics}, then $|\alpha_0\cap 
\alpha_2|=1$, $|\alpha_{n-2}\cap \alpha_n|=1$ and $\alpha_n$ is a meridian of $V$.  Let $\beta=\partial 
N(\alpha_0\cup \alpha_2)$. Since $|\alpha_0\cap \alpha_2|=1$ and $\alpha_2$ is a 
meridian of the handlebody $W$, then $\beta$ is a meridian of $W$, and it is a meridian of $V$ as well. Similarly, $\gamma=\partial 
N(\alpha_{n-2}\cup \alpha_n)$ is a meridian of $V$. So we have the following diagram. 

$
 \xymatrix{
& &\alpha_0 \ar@{-} [r]  \ar@/^1pc/[rr]^{\beta} & \alpha_1 \ar@{-} [r] & \alpha_2 \ar@{-} [r]  & \alpha_3 \ar@{. }[r] & \alpha_{n-2} \ar@{-} [r]  \ar@/^1pc/[rr]^{\gamma}& \alpha_{n-1} \ar@{-} [r] & \alpha_n.
}
$

Take the trivial compression body $S \times [0,1]$, where $S$ is the Heegaard surface in the standard Heegaard splitting. 
Let $W_0 = S[\beta]$ and $V_0 = S[\gamma]$, then $V_0 \subset V$ and $W_0 \subset W$, and $V_0 \cup_S W_0 \subset V \cup_S W$ can be embedded in $S^3$.
Note that the geodesic segment 
$[\alpha_2,\cdots,\alpha_{n-2}]$ has distance $n-4$ in curve graph 
$\mathcal{C}(S)$. By the triangle inequality,
 $$d_{\mathcal{C}(S)}(\beta,\gamma)\leq 
d_{\mathcal{C}(S)}(\beta,\alpha_2)+d_{\mathcal{C}(S)}(\alpha_2,\alpha_{n-2})+d_{
\mathcal{C}(S)}(\alpha_{n-2},\gamma)=1+(n-4)+1=n-2.$$

On the other hand, by Proposition \ref{one meridian}, $\beta$ is the unique meridian in $W_0$, and $\gamma$ is a unique meridian in $V_0$. Since 
$d_{\mathcal{C}(S)}(\alpha_0,\beta) = 1$ and 
$d_{\mathcal{C}(S)}(\gamma,\alpha_n) = 1$, then 
 $$d_{\mathcal{C}(S)}(\alpha_0,\alpha_n)\leq 
d_{\mathcal{C}(S)}(\alpha_0,\beta)+d_{\mathcal{C}(S)}(\beta,
\gamma)+d_{\mathcal{C}(S)}(\gamma,\alpha_n)\leq 
1+d_{\mathcal{C}(S)}(\beta,\gamma)+1.$$
 It follows that
 $$d_{\mathcal{C}(S)}(\beta,\gamma)\geq 
d_{\mathcal{C}(S)}(\alpha_0,\alpha_n)-2=n-2.$$
Then we have 
$d_{\mathcal{C}(S)}(V_0,W_0)= n-2$.   
\end{proof}

\begin{Prop}
For any integer $g \geq 3$, there exists a genus $g$ Heegaard splitting $V_0 \cup_S W_0$ with distance $n$, where $n$ is any positive integer divisible by 4. The $V_0$ and $W_0$ are compression bodies and $V_0 \cup_S W_0$ can be embedded in $S^3$.
\end{Prop}

\begin{proof}
 The geodesic $[\alpha_0,\alpha_1,\alpha_2]$ consists of the $\alpha_0$ as a 
meridian of $V$ and $\alpha_2$ as a meridian of $W$. Let 
$[\beta_0,\beta_1,\beta_2]$ be a geodesic in the curve graph $\mathcal{C}(S)$ 
such that $\beta_0=\alpha_2$ and $\beta_2=\alpha_0$, but $\beta_1\neq 
\alpha_1$, see Fig.~\ref{genus3}.

\begin{figure}[ht]
\vspace{-0.2cm}
\scalebox{.20}{\includegraphics[origin=c]{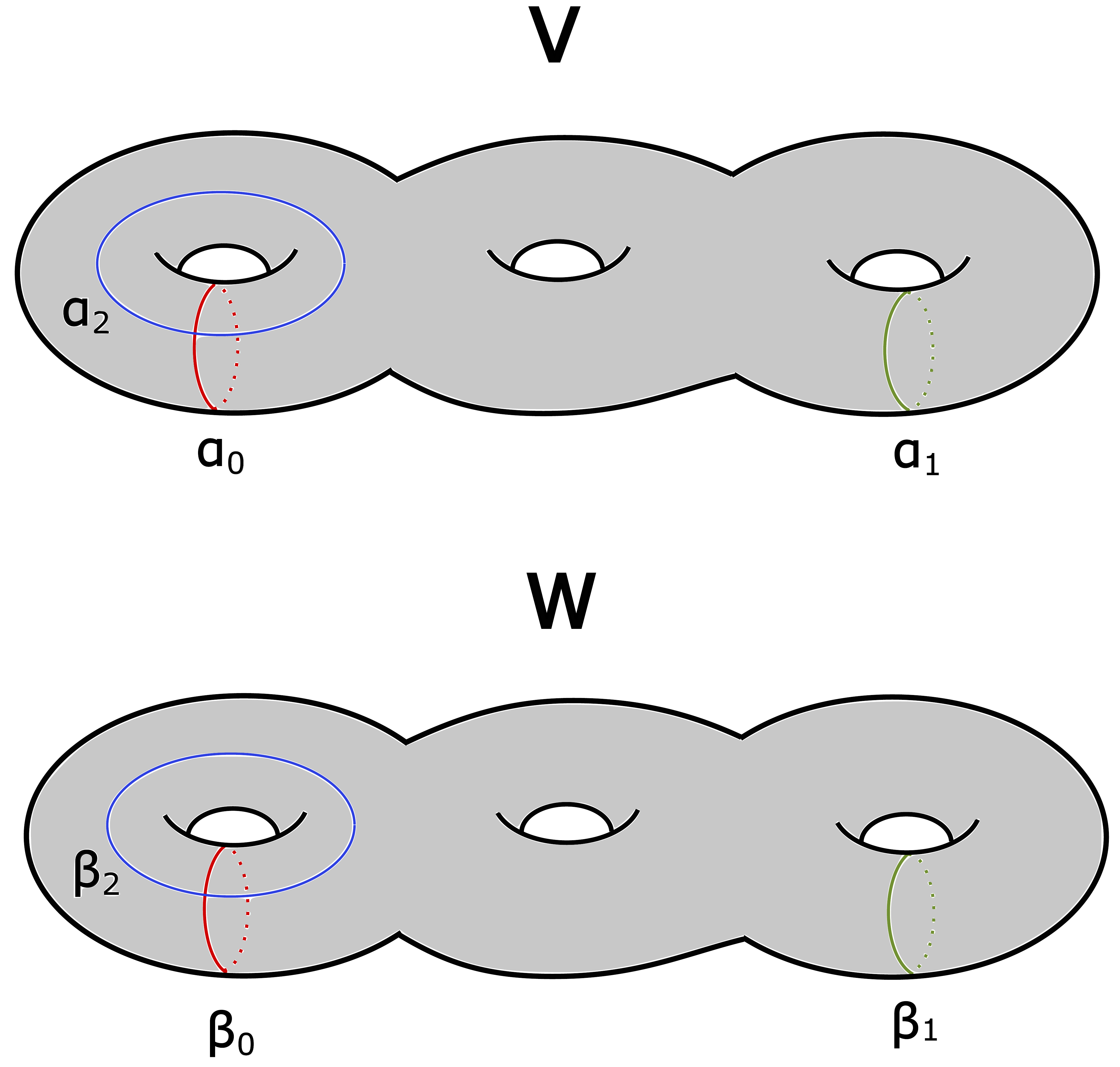}}
\vspace{-0cm}
\caption{ $S^3= V  \cup_S W$ is the standard Heegaard splitting. $\beta_0=\alpha_2$, $\beta_2=\alpha_0$ and $\beta_1\neq 
\alpha_1$.}
\label{genus3}
\end{figure}

By the same construction, one will be able to extend the geodesic $[\beta_0,\beta_1,\beta_2]$ to 
$[\beta_0,\beta_1,\beta_2,\beta_3,\beta_4]$ with $|\beta_0\cap \beta_2|=1$ and 
$|\beta_2\cap \beta_4|=1$.  Since $\alpha_0 = \beta_2$ and $\alpha_2 = \beta_0$, then $\partial N(\alpha_0\cup 
\alpha_2)=\partial N(\beta_0\cup \beta_2)$. Choose the geodesic 
$[\alpha_0,\alpha_1,\cdots,\alpha_{n-1},\alpha_{n}]$ as before with $n$ 
divisible by 4, and let $\beta=\partial 
N(\beta_2\cup \beta_4)$, $\gamma=\partial N(\alpha_{n-2}\cup \alpha_n)$, 
then 

$
 \xymatrix{
& &\alpha_0 \ar@{-} [r]\ar@{-} [d]^= & \alpha_1 \ar@{-} [r] \ar@{-} [d]^{\neq} & \alpha_2 \ar@{-} [r]  \ar@{-} [d]^=& \alpha_3 \ar@{. }[r] & \alpha_{n-2} \ar@{-} [r]  \ar@/^1pc/[rr]^{\gamma}& \alpha_{n-1} \ar@{-} [r] & \alpha_n\\
 \beta_4 \ar@/_1pc/[rr]_{\beta} & \ar@{-} [l] \beta_3 & \beta_2 \ar@{-} [l] & \ar@{-} [l] \beta_1& \ar@{-} [l] \beta_0&&&&
}
$
 is a path in the curve graph such that

 $$d_{\mathcal{C}(S)}(\beta,\gamma)\leq 
d_{\mathcal{C}(S)}(\beta,\beta_2=\alpha_0)+d_{\mathcal{C}(S)}(\alpha_0,\alpha_{
n-2})+d_{\mathcal{C}(S)}(\alpha_{n-2},\gamma)=1+(n-2)+1=n.$$

Let $W_0 = S[\beta]$ and $V_0 = S[\gamma]$. By Proposition \ref{one meridian}, $\beta$ is the unique meridian in $W_0$, and $\gamma$ is the unique meridian in $V_0$. Since $d_{\mathcal{C}(S)}(\beta_4,\beta) = 1$ and 
$d_{\mathcal{C}(S)}(\gamma,\alpha_n) = 1$, it follows that 
 $$d_{\mathcal{C}(S)}(\beta_4,\alpha_n)\leq 
d_{\mathcal{C}(S)}(\beta_4,\beta)+d_{\mathcal{C}(S)}(\beta,
\gamma)+d_{\mathcal{C}(S)}(\gamma,\alpha_n)\leq 
1+d_{\mathcal{C}(S)}(\beta,\gamma)+1.$$
Rearrange the inequality, we obtain
 $$d_{\mathcal{C}(S)}(\beta,\gamma)\geq 
d_{\mathcal{C}(S)}(\beta_4,\alpha_n)-2=(n+2)-2=n.$$
Then we have 
$d_{\mathcal{C}(S)}(V_0,W_0)= n$.
\end{proof}

To sum up, the previous two propositions prove Theorem \ref{compression} for all even natural numbers. In the following, we want to show it also holds for the odd natural number $n \geq 3$. 

\begin{Prop}
For any integer $g \geq 3$, there exists a genus $g$ Heegaard splitting $V_0 \cup_S W_0$ with distance $n-1$, where $n$ is any positive integer divisible by 4. The $V_0$ and $W_0$ are compression bodies and $V_0 \cup_S W_0$ can be embedded in $S^3$.
\end{Prop}

\begin{proof}
Using Lemma \ref{lem: geodesics}, we can find a geodesic $[\alpha_0,\alpha_1,\cdots,\alpha_{n},\alpha_{n+1}]$ such that the geodesic segment $[\alpha_0,\alpha_1,\cdots,\alpha_{n}]$ is of length $n$ divisble by 4, with $|\alpha_0\cap \alpha_2|=1$ and $|\alpha_{n-1}\cap \alpha_{n+1}|=1$. Let $\beta = \partial N(\alpha_0\cup \alpha_2)$, and $\gamma=\partial N(\alpha_{n-1}\cup \alpha_{n+1})$, then $\beta$ is the meridian of $W$ and $\gamma$ is a meridian of $V$. The reason is that $\alpha_2$ is a meridian of $W$ and $\alpha_{n-1}$ is a meridian of $V$. We can construct the compression bodies $W_0 = S[\beta]$ and $V_0 = S[\gamma]$. The diagram illustrates the path

$
 \xymatrix{
& &\alpha_0 \ar@{-} [r]  \ar@/^1pc/[rr]^{\beta} & \alpha_1 \ar@{-} [r] & \alpha_2 \ar@{-} [r]  & \alpha_3 \ar@{. }[r] & \alpha_{n-1} \ar@{-} [r] \ar@/^1pc/[rr]^{\gamma} & \alpha_{n} \ar@{-} [r] & \alpha_{n+1}.
}
$
Then, 
$$d_{\mathcal{C}(S)}(\beta,\gamma)\leq 
d_{\mathcal{C}(S)}(\beta,\alpha_2)+d_{\mathcal{C}(S)}(\alpha_2,\alpha_{n-1})+d_{
\mathcal{C}(S)}(\alpha_{n-1},\gamma)=1+(n-3)+1=n-1.$$

By Proposition \ref{one meridian}, $\beta$ is the unique meridian in $W_0$, and $\gamma$ is the unique meridian in $V_0$.  Since
$d_{\mathcal{C}(S)}(\alpha_0,\beta) = 1$ and 
$d_{\mathcal{C}(S)}(\gamma,\alpha_{n+1}) = 1$. It follows that 
 $$d_{\mathcal{C}(S)}(\alpha_0,\alpha_{n+1})\leq 
d_{\mathcal{C}(S)}(\alpha_0,\beta)+d_{\mathcal{C}(S)}(\beta,
\gamma)+d_{\mathcal{C}(S)}(\gamma,\alpha_{n+1})\leq 
1+d_{\mathcal{C}(S)}(\beta,\gamma)+1.$$
 Then we have 
 $$d_{\mathcal{C}(S)}(\beta,\gamma)\geq 
d_{\mathcal{C}(S)}(\alpha_0,\alpha_{n+1})-2=(n+1)-2=n-1.$$
Hence,
$d_{\mathcal{C}(S)}(V_0,W_0)= n-1$.
\end{proof}

\begin{Prop}
For any integer $g \geq 3$, there exists a genus $g$ Heegaard splitting $V_0 \cup_S W_0$ with distance $n+1$, where $n$ is any positive integer divisible by 4. The $V_0$ and $W_0$ are compression bodies and $V_0 \cup_S W_0$ can be embedded in $S^3$.
\end{Prop}

\begin{proof}
The proof is similar as the even case. Take the geodesic in the above proposition and extend the geodesic on the other end by distance 2.  The following path between $\beta$ and $\gamma$ realizes the distance. 

$
 \xymatrix{
& &\alpha_0 \ar@{-} [r]\ar@{-} [d]^= & \alpha_1 \ar@{-} [r] \ar@{-} [d]^{\neq} & \alpha_2 \ar@{-} [r]  \ar@{-} [d]^=& \alpha_3 \ar@{. }[r] & \alpha_{n-1} \ar@{-} [r]  \ar@/^1pc/[rr]^{\gamma}& \alpha_{n} \ar@{-} [r] & \alpha_{n+1}.\\
 \beta_4 \ar@/_1pc/[rr]_{\beta} & \ar@{-} [l] \beta_3 & \beta_2 \ar@{-} [l] & \ar@{-} [l] \beta_1& \ar@{-} [l] \beta_0&&&&
}
$
Similarly, we construct two compression bodies $V_0 = S[\gamma]$ and $W_0=S[\beta]$.
Then, $d_{\mathcal{C}(S)}(V_0, W_0)=2+(n-1)=n+1$. 
\end{proof}

The remaining cases that $n=0$ and $n=1$ are obvious. One can take two separating meridians in $V$ and $W$ that are same or disjoint. The resulting compression bodies has distance 0 or 1. With the preceding propositions, we complete the proof of Theorem \ref{compression}.

\section{\textbf{Proof of Theorem \ref{link}}}
\label{section: link complement}

In this section, we utilize Theorem \ref{compression} to show the exact distance of the link complement with a genus $g \geq 3$ Heegaard splitting. Let us recall Minsky, Moriah and Schleimer's work on the high distance knot. 
\begin{Th}\emph{(Minsky-Moriah-Schleimer \cite{MMS} Theorem 3.1)}\label{High}
 For any pair of integers $g>1$ and $n>0$, there is knot $K\subset S^3$ and a genus $g$ splitting $S \subset E(K)$ having distance greater than n. 
\end{Th}

In a nutshell, one can take the standard genus $g$ Heegaard splitting $S^3=V\cup_S W$, and let $D\subset V$ be a disk that cuts $V$ into a solid torus $X$ and a handlebody $Y$ of genus $g-1$.  Let $K_0$ be the core of $X$, then the complement $\widetilde{V_0}=V-N(K_0)$ is a compression body. It follows that $E(K_0)=\widetilde{V_0} \cup_S W$. The strategy is using a particular train track to construct a pseudo-Anosov map $\Phi$ that can be extended over $V$. Denote $V_n = \Phi^n(\widetilde{V_0})$ and $K_n=\Phi^n(K_0)\subset V\subset S^3$, then $V_n\cup_S W$ is a genus $g$ Heegaard splitting of the knot exterior $E(K_n)=S^3 - N(K_n)$. Then, 
$$d_{\mathcal{C}(S)}(\mathcal{D}_{V_n}, \mathcal{D}_W)\rightarrow \infty,$$
as $n\rightarrow \infty$. 

The iteration of $\Phi$ is a pseudo-Ansov map that can be extended over handlbody $V$. For any number $M$, there is a pseudo-Anosov map $\Phi$ such that $\Phi(\widetilde{V_0})\cup_S W$ is a Heegaard splitting of some knot exterior $E(K)$ with distance $d_{\mathcal{C}(S)}(\Phi(\widetilde{V_0}), W)> M.$

Our main goal in this section is to prove Theorem \ref{link}. 

\begin{proof} [Proof of Theorem \ref{link}]

For $n\geq 2$, using Theorem \ref{compression}, one can find the disk $D_1$ in $V$ and disk $D_2$ in $W$ such that the meridians $\partial D_1$ and $\partial D_2$ realize the distance $n$. By the construction, we can let $V_0 = S[\partial D_1]$ and $W_0 = S[\partial D_2]$, then the Heegaard distance $d_{\mathcal{C}(S)}(V_0,W_0) = n$. The two compression bodies are illustrated in Fig.~\ref{compressions}.

\begin{figure}[ht]
\vspace{-0.2cm}
\scalebox{.20}{\includegraphics[origin=c]{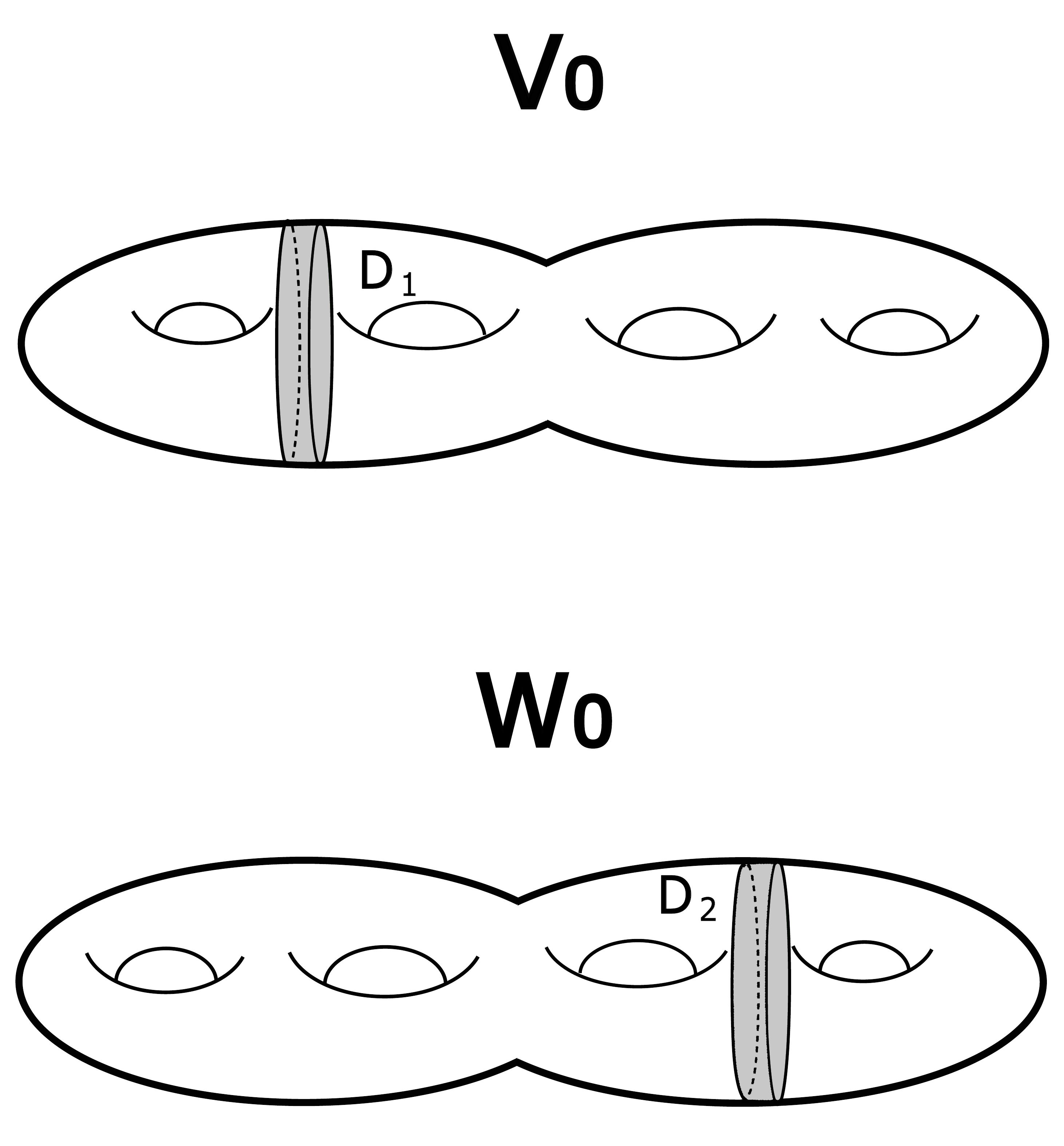}}
\vspace{0cm}
\caption{Compression body $V_0$ and $W_0$ are obtained by attaching a single 2-handle to $S \times [0,1]$ along the meridians $\partial D_1$ and $\partial D_2$.}
\label{compressions}
\end{figure}

Note that the figures are only for illustration, as $\partial D_1$ and $\partial D_2$ are supposed to be intersecting. In $V_0$, $W_0$, the meridians $\partial D_1$ and $\partial D_2$ divide the surface $S$ into two subsurfaces, one of which is a one-holed torus. Denote them as $S_1$, $S_2$ and $S_3$, $S_4$, and we assume $S_1$ and $S_3$ are one-holed tori, as shown in the \Cref{V0_glue,W0_glue}. In addition, we denote the corresponding negative boundaries as $\{F_1, F_2\}$ and $\{F_3, F_4\}$. As the genus $g(S) \geq 3$, then the genus $g(F_2)\geq 2$ and $g(F_4)\geq 2$. The proof is analogous to that of Proposition 3.1 in \cite{QZG} and Proposition 5.1 in \cite{IJK}.

The disk $D_1$ cuts $V_0$ into $F_1\times [0,1]$ and $F_2\times [0,1]$, and $D_2$ cuts $W_0$ into $F_3\times [0,1]$ and $F_4\times [0,1]$. In the compression body $V_0$, we identify $F_i = F_i \times \{1\}$, $S_i \cup D_1 = F_i \times \{0\}$ for $i=1,2$. Let $f_{F_i}: S_i\cup D_1 \rightarrow F_i$ be the homeomorphism such that $f_{F_i}(x \times \{0\})=x \times \{1\}$. The homeomorphism $f_{F_i}$ induces an isomorphism on the curve graphs, that is, 
$$d_{\mathcal{C}(F_i)}(f_{F_i}(\alpha), f_{F_i}(\beta)) = d_{\mathcal{C}(S_i\cup D_1)}(\alpha, \beta),$$
for any two essential simple closed curves $\alpha$ and $\beta$ on $S_i\cup D_1$. 

Let $l:S_i \rightarrow S_i\cup D_1$ be the inclusion map, then
$$d_{\mathcal{C}(S_i\cup D_1)}(l(\alpha), l(\beta)) \leq d_{\mathcal{C}(S_i)}(\alpha, \beta),$$
for any two essential simple closed curves $\alpha$ and $\beta$ on $S_i$. Define 
$$P_{F_i}=f_{F_i}\circ l \circ \pi_{S_i}: S\longrightarrow F_i$$
to be a map either between surfaces or the induced map between curve graphs, where $\pi_{S_i}$ is the subsurface projection.
Since $d_{\mathcal{C}(S)}(\partial D_1, \partial D_2) = n\geq 2$, then 
$$\diam_{\mathcal{C}(F_i)}(P_{F_i}(\partial D_2))\leq \diam_{\mathcal{C}(S_i)}(\pi_{S_i}(\partial D_2))\leq 2.$$
As a 3-manifold embedded in $S^3$, $V_0\cup W_0$ has two non-torus boundary components $F_2$ and $F_4$. The component of the complement of $V_0\cup W_0$ in $S^3$ with the boundary $F_2$ is a handlebody $X_{g-1}$ of genus $g-1$. It follows that $V_0\cup W_0\subset S^3 - X_{g-1}$, where $S^3 - X_{g-1}$ is a handlebody of genus $g-1$. It induces an inclusion of the disk graph, $\mathcal{D}(V_0\cup_S W_0, F_2)\subset \mathcal{D}(S^3 - X_{g-1}, F_2)$.

\begin{figure}[ht]
\vspace{-0.2cm}
\scalebox{.20}{\includegraphics[origin=c]{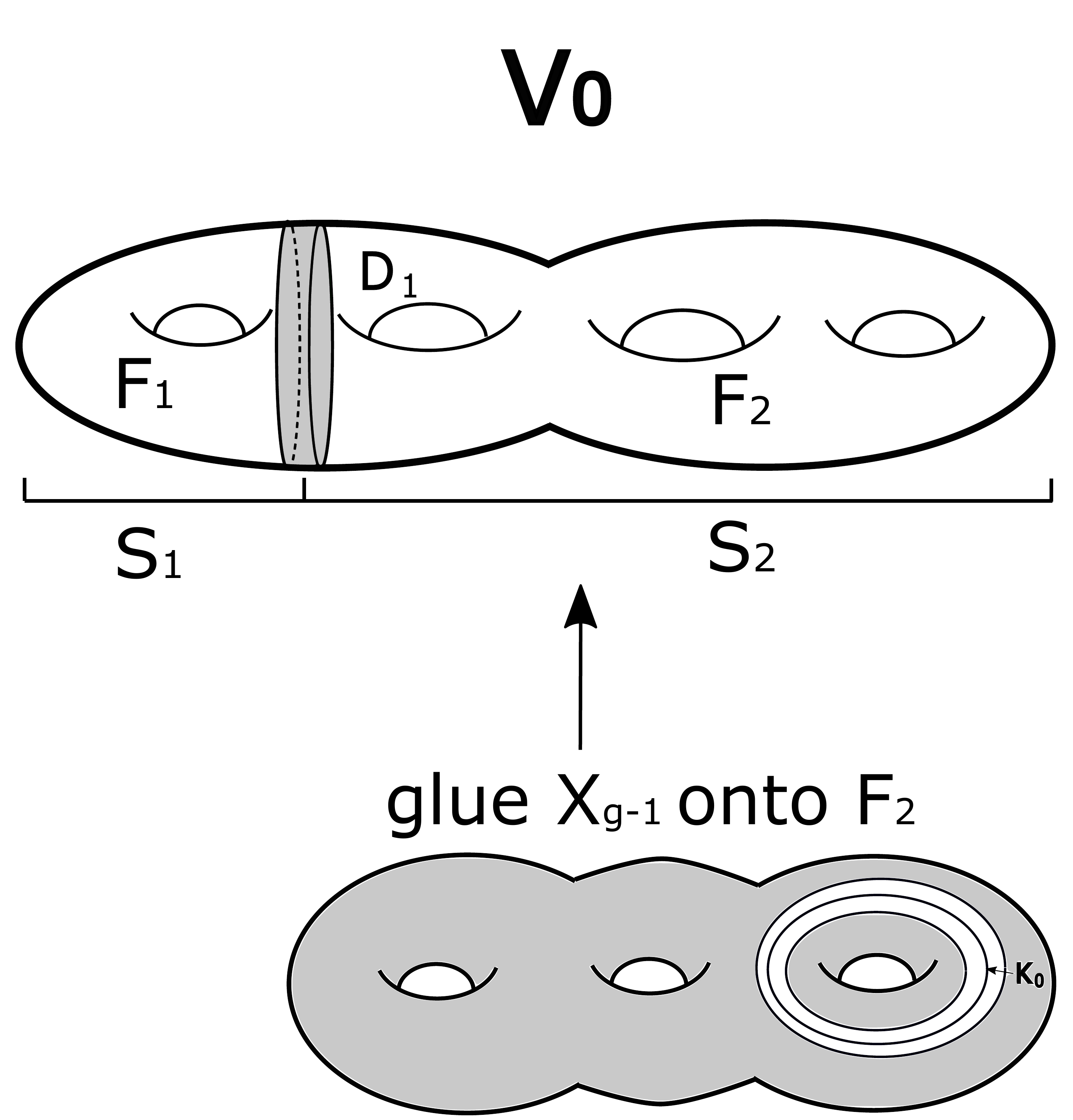}}
\vspace{0cm}
\caption{The meridian $\partial D_1$ cuts the positive (outer) boundary $S$ of the compression body $V_0$ into two subsurface $S_1$ and $S_2$. The gluing map is the standard identification of the boundary of handlebody $X_{g-1}$ onto $F_2$.}
\label{V0_glue}
\end{figure}

Inside the handlebody $X_{g-1}$, we take a core as a trivial knot $K_0$. Let $\widetilde{V_0}=X_{g-1} - K_0$. Note that $P_{F_2}(\partial D_2)$ has diameter 2 in the curve graph $\mathcal{C}(F_2)$, then 
$$d_{\mathcal{C}(F_2)}(P_{F_2}(\partial D_2), \mathcal{D}(S^3 - X_{g-1}, F_2)) \leq N_1 $$
for some constant $N_1$. Let $M$ be  the upper bound in the Bounded Geodedic Image Theorem \ref{Bounded}. By Theorem \ref{High}, there exists a pseudo-Anosov map $\Phi: F_2 \longrightarrow F_2$  such that $\Phi$ can be extended over the handlebody $X_{g-1}$, and 
$$d_{\mathcal{C}(F_2)}(\mathcal{D}(\Phi(\widetilde{V_0}), F_2), \mathcal{D}(S^3 - X_{g-1}, F_2)) > M + N_1 + 2.$$
By the triangle inequality, 
$$d_{\mathcal{C}(F_2)}(\mathcal{D}(\Phi(\widetilde{V_0}), F_2), P_{F_2}(\partial D_2)) > M.$$

\textbf{Claim 1.} The distance $d_{\mathcal{C}(S)}(V_{F_2}, W_0)=n$, where $V_{F_2} = V_0 \cup_{F_2}\Phi(\widetilde{V_0})$. 
\begin{proof}
Suppose not, then the distance $d_{\mathcal{C}(S)}(V_{F_2}, W_0)=k<n$. Since $W_0$ contains a unique disk $D_2$, then there is an essential disk $D\neq D_1$ in $V_{F_2}$ such that $d_{\mathcal{C}(S)}(\partial D,\partial D_2)=k<n$. It means there is a geodesic $\{\alpha_0=\partial D, \alpha_1, \cdots, \alpha_k = \partial D_2\}$ in curve graph $\mathcal{C}(S)$. Then for each $\alpha_i$, we have $\alpha_i \cap \partial D_1 \neq \varnothing$, for any $1\leq i \leq k$. Suppose there is some $\alpha_i$ such that $\alpha_i \cap \partial D_1=\varnothing$.  It follows that 

\begin{eqnarray*}
 n &=& d_{\mathcal{C}(S)}(\partial D_1, \partial D_2)\\
 &\leq& d_{\mathcal{C}(S)}(\partial D_1, \alpha_i) + d_{\mathcal{C}(S)}(\alpha_i, \partial D_2)\\
 &\leq& 1+(k-i)\leq k<n,
\end{eqnarray*}
which is a contradiction. Moreover, $\alpha_0 = \partial D$ is not in $S_1$. Otherwise, $D\subset F_1 \times [0,1]$, then $D$ is inessential. Hence, we show that $\alpha_i$ cuts  $S_2$ for all $0 \leq i \leq k$.

By the Bounded Geodesic Image Theorem \ref{Bounded}, we have
$d_{\mathcal{C}(S_2)}(\pi_{S_2}(\partial D), \pi_{S_2}(\partial D_2)) \leq M$. It implies that 
$d_{\mathcal{C}(S_2 \cup D_1)}(\pi_{S_2}(\partial D), \pi_{S_2}(\partial D_2))\leq M$ and $d_{\mathcal{C}(F_2)}(P_{F_2}(\partial D), P_{F_2}(\partial D_2))\leq M$. Assume that $D$ and $D_1$ intersect minimally. By the innermost disk argument, we can assume that $D\cap D_1$ has no loop components.

\textbf{Case i:} $|D\cap D_1|=0$. Since $D$ is not isotopic to $D_1$, then $P_{F_2}(\partial D)$ bounds an essential disk in the $\Phi(\widetilde{V_0})$, then 
$$d_{\mathcal{C}(F_2)}(P_{F_2}(\partial D), P_{F_2}(\partial D_2))\leq M,$$ which means that
$$d_{\mathcal{C}(F_2)}(\mathcal{D}(\Phi(\widetilde{V_0}), F_2), P_{F_2}(\partial D_2))\leq M.$$
It is contradictory to the choice of $\Phi(\widetilde{V_0})$.

\textbf{Case ii:} $|D\cap D_1| \neq 0$. Let $a$ be an outermost arc of $D\cap D_1$ on the disk $D$. It implies that $a$ and a subarc $b\subset \partial D_1$, bounds another disk $D'$ such that $D'\cap D_1 = a$. As we know that $D_1$ cuts $V_{F_2}$ into $\Phi(\widetilde{V_0})$ and $F_1 \times [0,1]$. Note that $D'\subset \Phi(\widetilde{V_0})$, then an essential simple closed curve in $P_{F_2}(\partial D)$ bounds an essential disk in $\Phi(\widetilde{V_0})$. By the \textbf{Case i}, it is impossible. 
\end{proof}

Next, let's look at the compression body $W_0$, which contains the unique essential disk $D_2$. The $D_2$ cuts $W_0$ into two submanifolds $F_3 \times [0,1]$ and $F_4 \times [0,1]$. $\partial D_2$ cuts $S$ into $S_3$ and $S_4$, where $S_3$ is a one-holed torus. Identify $F_i = F_i \times \{1\}$ and $S_i \cup D_2=F_i\times \{0\}$, for $i=3,4$. Similarly, let $f_{F_i}: S_i\cup D_2 \rightarrow F_i$ be the homeomorphism such that $f_{F_i}(x\times\{0\})=x\times\{1\}$. For any two essential simple closed curves $\alpha$ and $\beta$ on $S_i \cup D_2$, 
$$d_{\mathcal{C}(F_i)}(f_{F_i}(\alpha),f_{F_i}(\beta))=d_{\mathcal{C}(S_i\cup D_2)}(\alpha, \beta).$$
The isomorphism between the curve graphs is also denoted as $f_{F_i}$.  Let $l: S_i \rightarrow S_i\cup D_2$ be the inclusion map, then for any two essential simple closed curves $\alpha$ and $\beta$ on $S_i$,
$$d_{\mathcal{C}(S_i\cup D_2)}(l(\alpha),l(\beta)) \leq d_{S_i}(\alpha,\beta).$$
Let $\pi_{S_i}$ be the subsurface projection, we define $P_{F_i}=f_{F_i}\circ l \circ \pi_{S_i}: S \rightarrow F_i.$

\begin{figure}[ht]
\vspace{-0.2cm}
\scalebox{.20}{\includegraphics[origin=c]{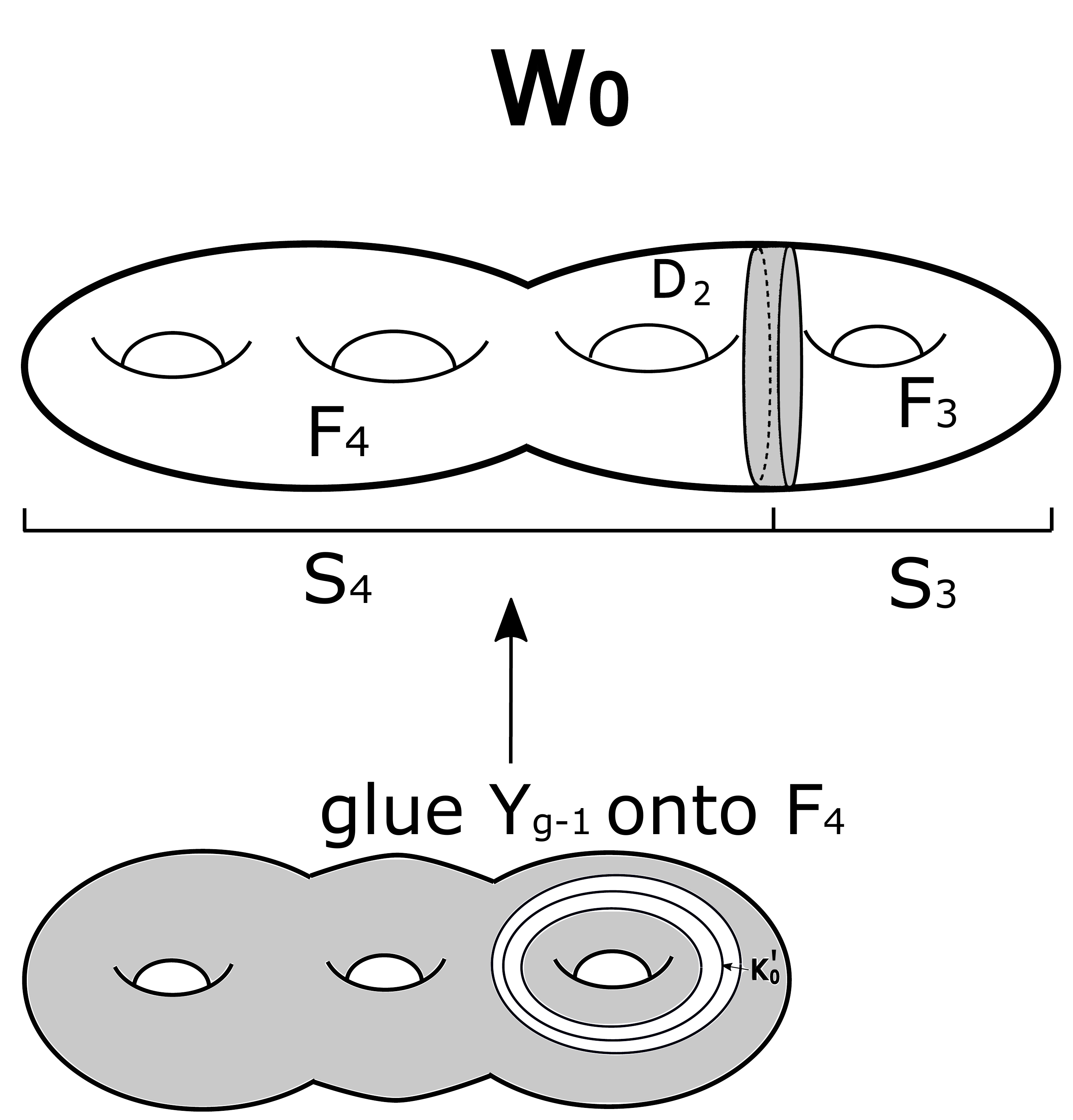}}
\vspace{0cm}
\caption{The meridian $\partial D_2$ cuts the positive (outer) boundary $S$ into two subsurface $S_3$ and $S_4$. The gluing map is the standard identification of the boundary of handlebody $Y_{g-1}$ onto $F_4$.}
\label{W0_glue}
\end{figure}

The Heegaard surface $S$ cuts $S^3$ into two handlebodies $V$ and $W$. Suppose that $V_{F_2} \subset V$, then $V_{F_2}$ has two torus boundary components. One is the $F_1$, and the other comes from the boundary of the regular neighborhood of the knot $\Phi(K_0)$. Both torus boundary components are incompressible in $V_{F_2}$.
Since $d_{\mathcal{C}(S)}(V_{F_2}, W_0) = n \geq 2$, and $W_0$ has the unique disk $D_2$, then the meridian $\partial D_2 = \partial S_4$ is disk-busting in $S$. Note that $V_{F_2}$ is not an $I$-bundle over some compact surface. Using Theorem \ref{Bounded disk}, we have $\diam_{\mathcal{C}(S_4)}(\pi_{S_4}(\mathcal{D}(V_{F_2}, S))) \leq 12$. It follows that 
 $$\diam_{\mathcal{C}(F_4)}(P_{F_4}(\mathcal{D}(V_{F_2}, S))) \leq \diam_{\mathcal{C}(S_4)}(\pi_{S_4}(\mathcal{D}(V_{F_2}, S)))\leq 12.$$
 The surface $F_4$ bounds two handlebodies in $S^3$. Denote the one that does not contain the 3-manifold $V_{F_2}\cup_{S} W_0$ as $Y_{g-1}$, then  $V_{F_2}\cup_{S} W_0$ lies in $S^3 - Y_{g-1}$. Since $P_{F_4}(\mathcal{D}(V_{F_2}, S))$ is bounded, there exists a constant $N_2$ such that
 $$d_{\mathcal{C}(F_4)}(P_{F_4}(\mathcal{D}(V_{F_2}, S)), \mathcal{D}(S^3 - Y_{g-1}, F_4)) \leq N_2.$$

Let a trivial knot $K'_0$ be a core of the handlebody $Y_{g-1}$ and $\widetilde{W_0}=Y_{g-1}-K'_0$. Once again, by Theorem \ref{High}, for the constant $M + N_2 + 12$, there is a pseudo-Anosov map $\Psi: F_{4} \longrightarrow F_{4}$ that can be extended over the handlebody $Y_{g-1}$, and it satisfies   
$$d_{\mathcal{C}(F_4)}(\mathcal{D}(\Psi(\widetilde{W_0}), F_4), \mathcal{D}(S^3 - Y_{g-1}, F_4)) > M + N_2 + 12.$$
Using the triangle inequality, we have 
$$d_{\mathcal{C}(F_4)}(\mathcal{D}(\Psi(\widetilde{W_0}), F_4), P_{F_4}(\mathcal{D}(V_{F_2}, S)) > M.$$

\textbf{Claim 2.} The distance $d_{\mathcal{C}(S)}(V_{F_2}, W_{F_4})=n$, where $W_{F_4}=W_0\cup_{F_4} \Psi(\widetilde{W_0})$. 
\begin{proof}
Suppose not, then the distance $d_{\mathcal{C}(S)}(V_{F_2}, W_{F_4})=k<n$. It means there is a geodesic $\{\alpha_0=\partial E, \alpha_1, \cdots, \alpha_k = \partial E'\}$ in curve graph $\mathcal{C}(S)$, where $E$ is an essential disk in $V_{F_2}$ and $E'$ is an essential disk in $W_{F_4}$. Note that $D_2$ is an essential disk in $W_{F_4}$. Then for each $\alpha_i$, we have $\alpha_i \cap \partial D_2 \neq \varnothing$, for any $0 \leq i \leq k-1$. Suppose there is some $\alpha_i$ such that $\alpha_i \cap \partial D_2=\varnothing$. 
It follows that 
\begin{eqnarray*}
 n&=&d_{\mathcal{C}(S)}(\partial D_1, \partial D_2)\\
 &\leq& d_{\mathcal{C}(S)}(\partial E, \partial D_2)\\
 &\leq& d_{\mathcal{C}(S)}(\partial E, \alpha_i)+d_{\mathcal{C}(S)}(\alpha_i, \partial D_2)\\
 &\leq& i + 1\leq k < n,
\end{eqnarray*}
which is a contradiction. Moreover, $\partial E'$ is not in $S_3$. Otherwise, $E'\subset F_3 \times [0,1]$, then $E'$ is inessential. Hence, we show that all $\alpha_i$ cut $S_4$ for $0\leq i \leq k$. By the Bounded Geodesic Image Theorem \ref{Bounded}, we have
$d_{\mathcal{C}(S_4)}(\pi_{S_4}(\partial E), \pi_{S_4}(\partial E'))\leq M$. It yields that
$d_{\mathcal{C}(S_4\cup D'_2)}(\pi_{S_4}(\partial E), \pi_{S_4}(\partial E'))\leq M$ and $d_{\mathcal{C}(F_4)}(P_{F_4}(\partial E), P_{F_4}(\partial E'))\leq M$.  Assume that $E'$ and $D_2$ intersect minimally. By the innermost disk argument, we can assume that $E' \cap D_2$ has no loop components.

\textbf{Case i:} $|E' \cap D_2|=0$. Since $E'$ is not isotopic to $D_2$ and $\partial E'$ is not in $S_3$, then $P_{F_4}(\partial E') = \partial E' \in \mathcal{D}(\Psi(\widetilde{W_0}), F_4)$.  We know that $\partial E \in \mathcal{D}(V_{F_2}, S)$, then the inequality 
$$d_{\mathcal{C}(F_4)}(\mathcal{D}(\Psi(\widetilde{W_0}), F_4), P_{F_4}(\mathcal{D}(V_{F_2}, S))\leq M$$ follows from 
$d_{\mathcal{C}(F_4)}(P_{F_4}(\partial E), P_{F_4}(\partial E'))\leq M$.
The inequality contradicts to the choice of $\Psi(\widetilde{W_0})$.

\textbf{Case ii:} $|E'\cap D_2|\neq 0$. Let $a$ be an outermost arc of $E'\cap D_2$ on the disk $E'$. It implies that $a$ and a subarc $b\subset \partial D_2$, bounds another disk $D''$ such that $D''\cap D_2=a$. As we know that $D_2$ cuts $W_{F_4}$ into $\Psi(\widetilde{W_0})$ and $F_3 \times [0,1]$. Note that $D'' \subset \Psi(\widetilde{W_0})$, then an essential simple closed curve in $P_{F_4}(\partial E')$ bounds an essential disk in $\Psi(\widetilde{W_0})$. Again, it is impossible by the  \textbf{Case i}. 
\end{proof}
Hence, the distance $d_{\mathcal{C}(S)}(V_{F_2}, W_{F_4})=n$, where $V_{F_2} \cup_S W_{F_4} = S^3 - N(K)$, and $K$ is a link of four components. This completes the proof of Theorem \ref{link}.
\end{proof}

\begin{Rem}
More link components can be added in the $\Psi(\widetilde{W_0})$, and it does not change the distance. 
\end{Rem}


\end{document}